\documentclass[a4paper,oneside,reqno]{amsart}
\numberwithin{equation}{section}
\usepackage{bbm}
\usepackage{comment}
\usepackage{verbatim}
\usepackage{mathrsfs}
\usepackage{mathtools}
\usepackage{latexsym}
\usepackage{epsfig}
\usepackage{amsmath}
\usepackage[usenames,dvipsnames]{xcolor}
\usepackage{bbm}
\usepackage{graphics}
\usepackage[utf8]{inputenc}
\usepackage{amssymb}
\usepackage{amsthm}
\usepackage[alphabetic]{amsrefs}
\usepackage{amsopn}
\usepackage{amscd}
\usepackage[all,knot]{xy}
\xyoption{all}
\usepackage{rotating}
\usepackage{tikz}
\usetikzlibrary{calc}
\usepgflibrary{shapes.geometric}
\usepgflibrary{shapes.misc}
\usetikzlibrary{positioning}
\usetikzlibrary{decorations}
\usetikzlibrary{decorations.pathreplacing}
\usepackage{hyperref}
\usepackage{tikz}
\usepackage{tikz-cd}
\usepackage{enumitem}
\mathtoolsset{showonlyrefs}

\newtheorem{thm}{Theorem}[section]
\newtheorem*{thm*}{Theorem}

\newtheorem*{cor*}{Corollary}
\newtheorem*{conj*}{Conjecture}

\newtheorem{prop}[thm]{Proposition}
\newtheorem{lem}[thm]{Lemma}
\theoremstyle{definition}
\newtheorem{dfn}[thm]{Definition}
\newtheorem{rmk}[thm]{Remark}

\newtheorem{exa}[thm]{Example}

\newcommand{\bbC}{{\mathbb C}}

\newcommand{\bbP}{{\mathbb P}}

\newcommand{\bbS}{{\mathbb S}}

\newcommand{\bbZ}{{\mathbb Z}}

\newcommand{\cA}{{\mathcal A}}
\newcommand{\cB}{{\mathcal B}}

\newcommand{\cD}{{\mathcal D}^b}
\newcommand{\cE}{{\mathcal E}}

\newcommand{\cH}{{\mathcal H}}
\newcommand{\cI}{{\mathcal I}}

\newcommand{\cM}{{\mathcal M}}
\newcommand{\cN}{{\mathcal N}}
\newcommand{\cO}{{\mathcal O}}
\newcommand{\cP}{{\mathcal P}}
\newcommand{\cQ}{{\mathcal Q}}

\newcommand{\cS}{{\mathcal S}}
\newcommand{\cT}{{\mathcal T}}

\newcommand{\cY}{{\mathcal Y}}

\newcommand{\hra}{\hookrightarrow}

\newcommand{\eps}{\varepsilon}

\newcommand{\Hom}{\mathrm{Hom}}
\newcommand{\Ext}{\mathrm{Ext}}

\newcommand{\Spec}{\mathrm{Spec}}

\newcommand{\rk}{\mathrm{rk}}

\newcommand{\id}{\mathrm{id}}
\newcommand{\ch}{\mathrm{ch}}

\newcommand{\Aut}{\mathrm{Aut}}

\newcommand{\Coh}{\mathrm{Coh}}

\newcommand{\im}{\mathrm{im}}
\newcommand{\cone}{\mathrm{cone}}

\newcommand{\HH}{\mathrm{HH}}
\newcommand{\HT}{\mathrm{HT}}

\renewcommand{\H}{\mathrm{H}}

\newcommand{\at}{\mathrm{At}}
\newcommand{\Xn}{X^{[n+1]}}

\newcommand{\ann}{\mathrm{ann}}

\renewcommand{\epsilon}{\varepsilon}
\renewcommand{\phi}{\varphi}

\DeclareMathOperator{\ob}{\mathsf{ob}}

\title{Hochschild cohomology and deformations of $\bbP$-functors}

\author{Ciaran Meachan}
\address{Ciaran Meachan, School of Mathematics and Statistics, University of Glasgow, Scotland}
\email{ciaran.meachan@glasgow.ac.uk}

\author{Theo Raedschelders}
\thanks{The second author is supported by an EPSRC postdoctoral fellowship EP/R005214/1.}
\address{Theo Raedschelders, School of Mathematics and Statistics,
University of Glasgow, Scotland
}
\email{theo.raedschelders@glasgow.ac.uk}

\begin{document}

\begin{abstract}
Given a split $\bbP$-functor $F:\cD(X) \to \cD(Y)$ between smooth projective varieties, we provide necessary and sufficient conditions, in terms of the Hochschild cohomology of $X$, for it to become spherical on the total space of a deformation of $Y$, and explain how the spherical twist becomes the $\bbP$-twist on the special fibre. These results generalise the object case, that is when $X$ is a point, which was studied previously by Huybrechts and Thomas, and we show how they apply to the $\bbP$-functor associated to the Hilbert scheme of points on a K3 surface. In the appendix we review and reorganise some technical results due to Toda, relating to the interaction of Atiyah classes, the HKR-isomorphism, and the characteristic morphism.
\end{abstract}

\maketitle

\tableofcontents

\section{Introduction}

If $\cD(Y)$ denotes the bounded derived category of coherent sheaves on a smooth complex projective variety $Y$ with trivial canonical bundle, then it is an enticing and very ambitious task to compute the group $\Aut(\cD(Y))$ of isomorphism classes of complex linear exact autoequivalences of $\cD(Y)$. The cases when $Y$ is a Calabi--Yau or hyperk\"ahler variety are of particular interest. Indeed, these two cases are governed by the complex $n$-sphere $\bbS^n$ and complex projective space $\bbP^n$, respectively, in the sense that the cohomology of the structure sheaf is isomorphic (as a ring) to the singular cohomology of $\bbS^n$ and $\bbP^n$ and mirror symmetry predicts the existence (in some appropriate limit) of Lagrangian fibrations over $\bbS^n$ and $\bbP^n$; see \cite{strominger1996mirror} for more details or \cite{gross2012mirror} for a survey.\footnote{We have also found \cite{verbitsky2010hyperkahler} to be particularly enlightening.}

The notion of a spherical object was introduced in the seminal paper of Seidel and Thomas \cite{seidel2001braid} and the resulting autoequivalences, called spherical twists, were studied in great detail. Huybrechts and Thomas \cite{huybrechts2006pobjects} completed the picture by introducing the notion of a $\bbP$-object and their induced autoequivalences, called $\bbP$-twists. In particular, they observed that a $\bbP$-object $\cE\in\cD(Y)$ which does not deform sideways in a one-dimensional family $\cY\to C$ over a smooth curve $C$ becomes a spherical object on the ambient space. Moreover, in such a situation, the spherical twist becomes the $\bbP$-twist on the special fibre $j:Y\hra\cY$.

Spherical and $\bbP^n$-objects were later generalised to spherical and $\bbP^n$-functors $F:\cA \to \cB$ between enhanced triangulated categories \cite{MR3692883,pfunctor}. Besides recovering the previous notions (by setting $\cA=\cD(\Spec(\bbC))$, and $\cB=\cD(Y)$), they unify various family versions of spherical and $\bbP^n$-objects and give rise to new derived autoequivalences, not coming from spherical or $\bbP^n$-objects. For a more detailed history, including ample references, see \cite{addington2011new}.

The aim of this note is to generalise the result of Huybrechts and Thomas to the setting of $\bbP$-functors. In doing so, one is faced with several choices, both in the methods used and in the level of abstraction. Since we want to highlight the analogy with the approach of Huybrechts and Thomas, we will stick to the setting of Fourier--Mukai functors between smooth complex projective varieties, instead of working with arbitrary enhanced triangulated categories. This has the advantage that the results are immediately applicable to geometric examples. However, to make up for this restriction, we have tried to emphasise throughout the abstract nature of our proofs, which are based on deformation theory and Hochschild cohomology. We are hopeful that it is possible to generalise our results to the setting of non-split $\bbP$-functors between enhanced triangulated categories \cite{pfunctor}, but have shied away from doing so in order to avoid some technical difficulties (see \S\ref{sec:con} for a brief discussion). Our main result, which we paraphrase here, is the following:

\begin{thm}[See Theorem \ref{thm:pbecomess} and Theorem \ref{th:main2}]
\label{thm:intro}
Let $F=\Phi_{\cP}:\cD(X) \to \cD(Y)$ be a split $\bbP^n$-functor between smooth projective varieties, and $j:Y \hookrightarrow \cY$ a one-parameter deformation of $Y$. Assume furthermore that $\HH^{2n+2}(X)=0$. Then the functor 
\begin{equation}
j_*F:\cD(X) \to \cD(\cY)
\end{equation}
is spherical if and only if as one deforms $Y$ to first order (through $\cY$), there is no generalised first order deformation of $X$ such that the functor $F$ deforms.

Moreover, in this case, the $\bbP$-twist $P_F$ and spherical twist $T_{j_*F}$ fit into a 2-commutative diagram: 
\begin{equation}
\begin{tikzcd}
\cD(Y)\ar[r,"j_*"]\ar[d,"P_F"'] & \cD(\cY)\ar[d,"T_{j_*F}"]\\ 
\cD(Y)\ar[r,"j_*"] & \cD(\cY).
\end{tikzcd}
\end{equation}
\end{thm}

The condition for the composition $j_*F$ to become spherical can be quantified in terms of explicit obstruction classes involving the second Hochschild cohomology group of $X$, which we discuss in \S2 (see Definition \ref{def:obsuv}). A rough slogan which was expounded in \cite{huybrechts2006pobjects} is that $\bbP^n$-objects should be thought of as hyperplane sections of spherical objects, and Theorem \ref{thm:intro} shows that a similar philosophy holds true in the context of $\bbP^n$-functors. Setting $X=\Spec(\bbC)$ in Theorem \ref{thm:intro} recovers the results of Huybrechts and Thomas. 

By now there are quite a few papers devoted to examples of $\bbP^n$-functors (for a non-exhaustive list, consult \cite{addington2011new,addington2015moduli,addington2015twists,pfunctor,cautis2012flops,krug2014nakajima,krug2017universal,markman2011integral,meachan2012derived}), so there are plenty of functors Theorem \ref{thm:intro} potentially applies to. Many of these functors are associated to  moduli problems of some kind, and indeed, the quintessential example of a $\bbP^n$-functor, which motivated its very definition, is the Fourier--Mukai functor 
\begin{equation}
\label{eq:quint}
F:=\Phi_{\cP}:\cD(X) \to \cD(\Xn),
\end{equation} 
where $X$ is a smooth projective K3 surface, $\Xn$ is the Hilbert scheme of $n+1$ points on $X$, and $\cP$ is the ideal sheaf of the universal closed subscheme of $X \times \Xn$. This functor was shown to be a $\bbP^n$-functor in \cite{addington2011new,markman2011integral}. In \S\ref{sec:hilb}, we use deep results by Markman \cite{markman2010integral,markman2011beauville} and Markman--Mehrotra \cite{markman2011integral} to show that Theorem \ref{thm:intro} applies to the functor $F$, see Theorem \ref{thm:hilb} for a precise statement. Of course, we expect that the theorem applies to other interesting $\bbP^n$-functors as well.

Finally, let us say a few words about the proof. It is perhaps  surprising that in Theorem \ref{thm:intro}, even though we only consider an honest geometric deformation $\cY$ of $Y$, the condition for $j_*F$ to become spherical involves generalised deformations of $X$, which one can think of as describing not only the geometric, but also the noncommutative and gerby deformations.\footnote{This can be made precise through the Hochschild--Kostant--Rosenberg isomorphism, see \S\ref{sec:hhdefs}.} The underlying reason for this is the following: even though Hochschild cohomology is not functorial, any Fourier--Mukai functor induces a type of correspondence \eqref{eq:corr} on Hochschild cohomology groups. After applying HKR to both sides, this correspondence allows one to compare the commutative, noncommutative, and gerby deformations of $X$ and $Y$, but there is no a priori reason why the different types of deformations should line up, and indeed they typically do not.\footnote{In the single object case, where $X=\Spec(\bbC)$ is a point, this phenomenon is not visible since a point has no non-trivial deformations, commutative or otherwise.} Accepting the fact that one also needs to take into account these generalised deformations of $X$, it turns out to be possible to lift the proof of Huybrechts and Thomas to the functor setting. 

\subsection*{Acknowledgements} 
This paper owes a great deal to the preprint \cite{markman2011integral}, and it can be considered as an extended elaboration on some of the many new, and largely unexplored, ideas contained in the work of Markman and Mehrotra. The first author is very grateful to Daniel Huybrechts and Richard Thomas for encouragement, support and expert advice. Both authors would like to thank Andreas Krug, Wendy Lowen, Emanuele Macr\`i, and Paolo Stellari for helpful comments on a preliminary version of this paper.

\subsection*{Conventions} 
Throughout, we work over the complex numbers $\bbC$. Unless explicitly stated otherwise, all varieties will be smooth and projective over $\bbC$, and we will denote by $\cD(X)$ the bounded derived category of coherent sheaves on $X$. For a Fourier--Mukai functor $F=\Phi_\cP:\cD(X) \to \cD(Y)$ with kernel $\cP \in \cD(X \times Y)$, we will denote the kernel of its left adjoint $L$ by $\cP_L$, and the kernel of its right adjoint $R$ by $\cP_R$. We denote by $\cO_{\Delta}:=\Delta_*\cO_X$, for $\Delta:X \to X \times X$ the diagonal embedding.  All functors are implicitly derived. 

\section{Hochschild cohomology and Fourier--Mukai transforms}
\label{sec:hhfm}
In \cite{toda2009deformations}, Toda describes how deformations of smooth projective varieties interact with Fourier--Mukai equivalences. His techniques actually allow one to understand how arbitrary Fourier--Mukai functors interact with deformations, and this is what we explain in this section. We believe this is well-known to experts, but for lack of a reference we spell it out and refer to Appendix \ref{app} for further details.

Throughout this section, let $F=\Phi_\cP:\cD(X) \to \cD(Y)$ denote a Fourier--Mukai functor with kernel $\cP \in \cD(X \times Y)$. 

\subsection{Hochschild cohomology and deformations}
\label{sec:hhdefs}
Recall that the Hochschild cohomology of $X$ was defined in \cite{MR1390671} as 
\begin{equation}
\label{def-hh}
\HH^{\bullet}(X) :=\Ext^{\bullet}_{X \times X}(\cO_{\Delta},\cO_{\Delta}).
\end{equation}
In \cite{MR2238922}, Lowen and Van den Bergh develop an infinitesimal deformation theory for abelian categories, and in \cite{MR2183254} it is shown that these deformations are controlled by a suitable notion of Hochschild cohomology for abelian categories. They furthermore show \cite[Corollary 7.8.2]{MR2183254} that the Hochschild cohomology of the abelian category $\Coh(X)$ of coherent sheaves on the smooth projective variety $X$ agrees with \eqref{def-hh}, and obtain the following theorem, which we paraphrase since we won't need the details.

\begin{thm}\cite[Theorem 3.1]{MR2183254}
\label{th:lvdb}
The infinitesimal deformations of the abelian category $\Coh(X)$ are classified by $\HH^2(X)$.
\end{thm}

The link with the usual geometric deformations of $X$ is provided by the Hochschild--Kostant--Rosenberg (HKR) theorem \cite{caldararu2005mukai2}, which implies that there is an isomorphism
\begin{equation}
I_X:\HT^{\bullet}(X) \xrightarrow{\simeq} \HH^{\bullet}(X),
\end{equation}
where
\begin{equation}
\HT^{\bullet}(X):=\bigoplus_{i+j=\bullet} \H^i(X,\wedge^j \cT_X)
\end{equation}
denotes the tangent cohomology of $X$, and $\cT_X$ is the tangent bundle of $X$. In particular, there is an isomorphism
\begin{equation}
\label{eq:hkr}
I_X:\H^2(X,\cO_X) \oplus \H^1(X,\cT_X) \oplus \H^0(X,\wedge^2\cT_X) \xrightarrow{\simeq} \HH^2(X).
\end{equation}
The middle cohomology group is well-known to classify the infinitesimal deformations of $X$ as a scheme. For a triple $(\alpha,\beta,\gamma)$ corresponding to class $u \in \HH^2(X)$ under the isomorphism \eqref{eq:hkr}, Toda \cite{toda2009deformations} explicitly constructs a $\bbC[\epsilon]/(\epsilon^2)$-linear abelian category $\Coh(X,u)$. We will not need the details of the construction, so let us just briefly mention the idea. In a first step, one can use $(\beta,\gamma) \in \H^1(X,\cT_X) \oplus \H^0(X,\wedge^2\cT_X)$ to deform $\cO_X$ to a sheaf of (typically noncommutative) $\bbC[\epsilon]/(\epsilon^2)$-algebras $\cO_X^{(\beta,\gamma)}$ on X. By representing $\alpha \in \H^2(X,\cO_X)$ as a \v{C}ech 2-cocycle $\{\alpha_{ijk}\}$, one constructs a \v{C}ech 2-cocycle $\widetilde{\alpha}=\{1-\epsilon \alpha_{ijk}\}$ with values in $Z(\cO_{X}^{(\beta,\gamma)})^{\times}$, the invertible elements of the center of $\cO_X^{(\beta,\gamma)}$. One can then define 
\begin{equation}
\Coh(X,u):=\Coh(\cO_X^{(\beta,\gamma)},\widetilde{\alpha})
\end{equation}
to be the category of $\widetilde{\alpha}$-twisted coherent modules over $\cO_X^{(\beta,\gamma)}$. In \cite{van2017non} it is shown that $\Coh(X,u)$ is indeed an infinitesimal deformation of $\Coh(X)$ corresponding to $u$ in the sense of Theorem \ref{th:lvdb}. Finally one sets 
\begin{equation}
\cD(X,u):=\cD(\Coh(X,u))
\end{equation}
to be the bounded derived category of this infinitesimal deformation of $\Coh(X)$.

\subsection{Functoriality of Hochschild cohomology}

It is well known that Hochschild cohomology only has limited functoriality properties \cite[\S 4.3]{keller2003derived}. In particular, for an arbitrary Fourier--Mukai functor, there is no induced morphism on Hochschild cohomology in either direction. There is however a type of correspondence which turns out to suffice for our needs.

Define a functor 
\begin{equation}
\cP \ast -:\cD(X \times X) \to \cD(X \times Y)\;;\,\cE \mapsto p_{13*}(p_{12}^*\cE \otimes p_{23}^*\cP),
\end{equation}
where $p_{ij}$ denotes the projection from $X \times X \times Y$ onto the corresponding factors. Similarly, one can define a functor $- \ast \cP:\cD(Y \times Y) \to \cD(X \times Y)$, and it is not hard to check that $\cP \ast \cO_{\Delta_X} \cong \cO_{\Delta_Y} \ast \cP \cong \cP$. Therefore, using \eqref{def-hh}, these functors give rise to the following correspondence:
\begin{equation}
\label{eq:corr}
\begin{tikzcd}[column sep=small]
& \Hom_{X \times Y}(\cP,\cP[\bullet]) & \\
\HH^{\bullet}(X) \ar{ru}{\cP \ast -} \ar[dashed,no head]{rr} & & \HH^{\bullet}(Y) \ar[swap]{ul}{- \ast \cP}
\end{tikzcd}
\end{equation}
allowing us to compare Hochschild classes on $X$ and $Y$ in the space $\Hom_{X \times Y}(\cP,\cP[\bullet])$, even though there is no induced morphism between $\HH^{\bullet}(X)$ and $\HH^{\bullet}(Y)$. 

\begin{dfn}
\label{def:obsuv}
For $u \in \HH^2(X)$ and $v \in \HH^2(Y)$, we define 
\begin{equation}
\ob_{\cP}(u,v):=-\cP \ast u + v \ast \cP \in \Hom_{X \times Y}(\cP,\cP[2])
\end{equation}
and say it is the \textit{obstruction class associated to the pair $(u,v)$}.
\end{dfn}

\begin{rmk}
If $F$ is fully faithful, then one checks that the morphism $\cP\ast-$ in \eqref{eq:corr} is an isomorphism, so by inverting it we obtain a morphism $\phi:\HH^{\bullet}(Y) \to \HH^{\bullet}(X)$, and $\ob_{\cP}(u,v)$ measures the difference between $u$ and $\phi(v)$.
\end{rmk}

\subsection{The characteristic morphism}

By \cite[Proposition 6.10]{MR2183254}, we know that for any object $\cE \in \cD(X)$ one can define a characteristic morphism 
\begin{equation}
\chi_{\cE}:\HH^2(X) \to \Hom_{X}(\cE,\cE[2]).
\end{equation}
If we regard the elements of $\HH^2(X)$ as natural transformations between the functors $\mathrm{id}_{\cD(X)}$ and $[2]$, then evaluating them on $\cE$ defines $\chi_{\cE}$; alternatively, if we view $\cE$ as an object in $\cD(\Spec(\bbC)\times X)$ then $\chi_\cE$ is the degree two part of the functor $-\ast\cE$. The characteristic morphism has a deformation theoretic interpretation given by the following theorem, which is the key to understanding deformations of Fourier--Mukai transforms. 

\begin{thm}\cite{MR2474321,toda2009deformations}
\label{thm:todalowen}
For $u \in \HH^2(X)$, the image $\chi_{\cE}(u)$ is exactly the obstruction to lifting $\cE$ to a perfect object of $\cD(X,u)$.
\end{thm}

We now use the characteristic morphism to construct a map
\begin{equation}
\begin{tikzcd}[row sep=tiny]
\widetilde{\chi}:\HH^2(X) \times \HH^2(Y) \ar{r}{K} & \HH^2(X \times Y) \ar{r}{\chi_{\cP}} & \Hom(\cP,\cP[2]),
\end{tikzcd}
\end{equation}
where $K$ is induced by the K\"unneth isomorphism for Hochschild cohomology. More precisely, we choose $K$ as follows
\begin{equation}
\label{eq:choicek}
K(u,v):=I_{X \times Y}(-p_X^*(\alpha,-\beta,\gamma)+p_Y^*(\alpha',\beta',\gamma')),
\end{equation}
where $p_X$ and $p_Y$ denote the projections from $X \times Y$ onto the factors, $I_X(\alpha,\beta,\gamma)=u$ and $I_Y(\alpha',\beta',\gamma')=v$. The reason for this choice will become clear in the next section.

\subsection{Deformations and Fourier--Mukai functors}

With this setup, we can now state the following theorem, which explains how deformations interact with Fourier--Mukai transforms.

\begin{thm}
\label{thm:def}
For $u \in \HH^2(X)$ and $v \in \HH^2(Y)$, the functor $F=\Phi_{\cP}:\cD(X) \to \cD(Y)$ lifts to a functor 
\begin{equation}
F_{u,v}:\cD(X,u) \to \cD(Y,v)
\end{equation}
if and only if $0=\ob_{\cP}(u,v) \in \Hom_{X \times Y}(\cP,\cP[2])$.
\end{thm}
\begin{proof}
We first check that $\ob_{\cP}(u,v)=\widetilde{\chi}(u,v)$. By \eqref{eq:choicek}, this reduces to checking that
\begin{align}
\label{eq:hkrprod'}
\cP \ast u &= \chi_{\cP}(I_{X \times Y}p_X^*(\alpha,-\beta,\gamma)),\\
\label{eq:hkrprod}
v\ast \cP &= \chi_{\cP}(I_{X \times Y}p_Y^*(\alpha',\beta',\gamma')).
\end{align}
Let us focus on  \eqref{eq:hkrprod} (the other case is similar): \begin{align}
v \ast \cP &= p_Y^*(\alpha',\beta',\gamma') \circ \exp(\at(\cP)) \\
&=\chi_{\cP}(I_{X \times Y}(p_Y^*(\alpha',\beta',\gamma')))
\end{align}
where the first equality uses \eqref{eq:todasquare}, and the second one uses Proposition \ref{prop:char}.

By Theorem \ref{thm:todalowen}, we find that $\ob_{\cP}(u,v)=0$ if and only if $\cP$ lifts to a perfect object $\cP_{u,v} \in \cD(X \times Y,K(u,v))$. It now suffices to show that $\cP_{u,v}$ can be used to define a functor $F_{u,v}:\cD(X,u) \to \cD(Y,v)$ lifting $F$. Toda explains how to do this in \cite[p.208]{toda2009deformations}, which one can follow word for word by our choice \eqref{eq:choicek} of $K$.
\end{proof}

The difficulty with applying this theorem in an actual example is that it is typically hard to compute $\ob_{\cP}(-,-)$. As we will see in the next section however, these classes show up naturally when studying $\bbP$-functors.

\section{Deforming $\bbP$-functors}

The aim of this section is to use the obstruction classes from Theorem \ref{thm:def} to give a generalisation of \cite[Proposition 1.4]{huybrechts2006pobjects} which gives conditions allowing one to obtain spherical objects on (and hence autoequivalences of the derived category of) the total space of a deformation of $Y$ from $\bbP^n$-objects on $Y$.

\subsection{$\bbS$-functors and $\bbP$-functors}

We first recall the definitions of a split spherical and a split $\bbP$-functor. 

\begin{dfn}\cite[\S 6.1]{cautis2012flops} 
\label{def:sph}
A Fourier--Mukai functor $F=\Phi_{\cP}:\cD(X) \to \cD(Y)$ is a split $\bbS^d$-functor if the following two conditions are satisfied:
\begin{enumerate}
\item\label{cond-1} $\cP_R \simeq \cP_L[-d]$.
\item\label{cond-2} $\cP_R \ast \cP \simeq \cO_{\Delta} \oplus \cO_{\Delta}[-d]$.
\end{enumerate}
\end{dfn}

\begin{dfn}\cite[Definition 4.1\footnote{This is a special case of Addington's definition where $H=[-2]$. It would be interesting to investigate what happens for arbitrary $H\in\Aut(\cD(X))$.}]{addington2011new}
\label{def:nickPn}
A functor $F=\Phi_{\cP}:\cD(X) \to \cD(Y)$ is a split $\bbP^n$-functor if the following three conditions are satisfied:
\begin{enumerate}
\item \label{cond-11} $\cP_R \simeq \cP_L[-2n]$.
\item \label{cond-22} There exists an isomorphism 
\begin{equation}
\label{eq:gammiso}
\gamma=(\gamma_i)_{i=0}^{n}: \bigoplus_{i=0}^{n}\cO_{\Delta}[-2i]\to \cP_R \ast \cP 
\end{equation}
\item \label{cond-33} The morphism
\begin{equation}
\tilde{\gamma_1}:\cP_R \ast \cP \ast \cO_{\Delta} \xrightarrow{\cP_R \ast \cP \ast \gamma_1} \cP_R \ast   \cP \ast \cP_R \ast \cP[2] \xrightarrow{\cP_R \ast \epsilon \ast \cP[2]} \cP_R \ast \cP[2]
\end{equation}
becomes an isomorphism after taking $\cH^i$, for $i=2k$ and $k=0,\ldots,n-1$.
\end{enumerate}
\end{dfn}

\begin{exa}
If $X=\Spec(\bbC)$ and $F$ is a split $\bbP^n$-functor, then $\cE:=F(\bbC) \in \cD(Y)$ is a $\bbP^n$-object, as introduced in \cite{huybrechts2006pobjects}.
\end{exa}

\begin{rmk}
For clarity of exposition and because this is all we need for our main example of interest in Section \ref{sec:hilb}, we only discuss split $\bbS$-functors and $\bbP$-functors. For this reason, we usually leave out the word `split', and simply refer to $\bbS$-functors and $\bbP$-functors. For a brief discussion of the non-split setting, as considered in \cite{MR3692883,pfunctor}, we refer to Section \ref{sec:con}.
\end{rmk}

\subsection{Atiyah classes and obstructions}
\label{sec:atiyah}
If $\Delta\coloneqq\Delta_Y\subset Y\times Y$ is the diagonal then there is an exact sequence:
\begin{equation}
\label{eq:atiyah}
0 \to \cI_{\Delta}/\cI_{\Delta}^2 \to \cO_{Y\times Y}/\cI_{\Delta}^2\to \cO_{\Delta} \to 0,
\end{equation}
which induces a triangle whose boundary map is the universal Atiyah class of $Y$: 
\begin{equation}
\label{alpha}
\at_Y:\cO_\Delta\to\Omega_\Delta[1],
\end{equation}
where $\Omega_{\Delta}:=\cI_{\Delta}/\cI_{\Delta}^2$. For any object $\cE\in\cD(Y)$, applying $p_{2*}(p_1^*\cE \otimes -)$ to \eqref{eq:atiyah}, we recover the classical Atiyah class $\at(\cE):\cE\to\cE\otimes\Omega_Y[1]$. 

For an object $\cP\in\cD(X\times Y)$, the isomorphism $\Omega_{X\times Y}\simeq p_X^*\Omega_X\oplus p_Y^*\Omega_Y$ means that the Atiyah class: 
\begin{equation}
\at(\cP):\cP\to\cP\otimes\Omega_{X\times Y}[1],
\end{equation}
decomposes $\at(\cP)=\at_X(\cP) + \at_Y(\cP)$ 
into `partial' Atiyah classes, where:
\begin{equation}
\label{partial-atiyah}
\at_X(\cP):\cP\to\cP\otimes p_X^*\Omega_X[1]
\qquad\text{and}\qquad
\at_Y(\cP):\cP\to\cP\otimes p_Y^*\Omega_Y[1].
\end{equation}

Suppose $X$ and $Y$ are smooth complex projective varieties and assume $\cY \to C$ is a smooth family over a smooth curve $C$ with distinguished fibre $j:Y\hra\cY$. Then we denote the Kodaira--Spencer class of this family by $\kappa(\cY)\in\H^1(Y,\cT_Y)$. For a Fourier--Mukai functor $F=\Phi_{\cP}:\cD(X)\to\cD(Y)$, the morphism
\begin{equation}
\ob(\cP)\coloneqq(1_\cP\otimes p_Y^*\kappa(\cY))\circ \at_Y(\cP)\in\Hom_{X\times Y}(\cP,\cP[2]),
\end{equation}
is known to be the global obstruction class to deforming $\cP$ sideways (to first order) to neighbouring fibres in the trivially extended family 
\begin{equation}
\begin{tikzcd}
X \times Y \ar[hook]{r}{i} \ar{d} & X \times \cY \ar{d} \\ 
X \times \Spec(\bbC) \ar[hook]{r} & X\times C
\end{tikzcd}
\end{equation} 
where $i=\id_X \times j$. See \cite[IV.3.1.8]{illusie1971complexe} for the original statement and much more besides, or \cite[Proposition 3.8]{buchweitz1998atiyah} for an elegant summary. In fact, this can also be deduced from Theorem \ref{thm:todalowen} together with Propositions \ref{prop:todalem5.8} and \ref{prop:char}. For more on Atiyah classes and their exponential versions, we refer to Appendix \ref{app}.

\subsection{Deforming $\bbP$-functors to $\bbS$-functors}

We are now ready to prove the promised generalisation of \cite[Proposition 1.4]{huybrechts2006pobjects}. 

\begin{lem}
\label{prop:notinimage}
Let $F=\Phi_{\cP}:\cD(X) \to \cD(Y)$ be a $\bbP^n$-funcor and let $v \in\HH^2(Y)$. Then the following are equivalent:
\begin{enumerate}
\item\label{cond-111} for all $0\leq k \leq n-1$, the morphisms 
\begin{equation}
\cH^{2k}(\cP_R \ast v \ast \cP):\cH^{2k}(\cP_R \ast \cO_{\Delta} \ast \cP) \to \cH^{2k+2}(\cP_R \ast  \cO_{\Delta} \ast \cP)
\end{equation}
are isomorphisms.
\item\label{cond-222} $\forall u \in \HH^2(X): \ob_{\cP}(u,v) \neq 0$.
\end{enumerate}
\end{lem}
\begin{proof}
By \eqref{eq:gammiso}, after conjugating by $\gamma$, any morphism $f \in \Hom(\cP_R \ast \cP,\cP_R \ast \cP[2])$ can be represented as a matrix $f=(f_{ij})_{i,j=1}^{n+1}$, for $f_{ij}\in\HH^{2j-2i+2}(X)$ (of course $\HH^{<0}(X)=0$). Since 
\begin{equation}
\begin{tikzcd}[row sep=tiny]
\label{eq:vastp}
\Hom(\cP,\cP[2]) \ar{r}{\sim} & \Hom(\cO_{\Delta},\cP_R\ast \cP[2])\ar{r}{\sim} & \HH^2(X) \oplus \HH^0(X)\\
v \ast \cP \ar[mapsto]{rr} & & (v_1,v_2) 
\end{tikzcd}
\end{equation}
for a morphism of the form $f=\cP_R \ast v \ast \cP$, we have
\begin{equation}
\label{eq:descrf}
f_{ij}=
\begin{cases}
v_1 & \text{ if } i=j \\
v_2 & \text{ if } j=i-1 \\
0 & \text{ otherwise } 
\end{cases}
\end{equation} 
In particular, we see that condition \eqref{cond-111} is satisfied if and only if $0 \neq v_2 \in \HH^0(X) \cong \bbC$. We now claim there is a commuting diagram  
\begin{equation}
\label{comsq}
\begin{tikzcd}[column sep=huge]
\HH^2(X)\ar[hookrightarrow]{rr}{\cP\ast -} \ar[hookrightarrow]{d}{(\id,0)}&&\Hom(\cP,\cP[2])\ar{d}{\wr} \\
\HH^2(X) \oplus \HH^0(X) \ar{rr}{\Hom(\cO_{\Delta}[-2],\gamma)}[swap]{\sim} && 
\Hom(\cO_{\Delta},\cP_R\ast \cP[2]) .
\end{tikzcd}
\end{equation} 
Indeed, the unit $\eta:\cO_{\Delta} \to \cP_R\ast \cP$ induces a map:
\begin{equation}
\Hom(\cO_{\Delta}[-2],\eta):\Hom(\cO_{\Delta},\cO_{\Delta}[2]) \to \Hom(\cO_{\Delta},\cP_R\ast \cP[2])\;;\,\xi\mapsto\eta\circ\xi,\end{equation} 
which corresponds to the top-right composition in the diagram. Moreover, the isomorphism
\begin{equation}
\Hom(\cO_{\Delta},\gamma): \Hom(\cO_{\Delta},\bigoplus_{i=0}^n \cO_{\Delta}[-2i]) \to\Hom(\cO_{\Delta},\cP_R \ast \cP)
\end{equation}
maps the inclusion $i:\cO_{\Delta} \to \bigoplus_{i=0}^n \cO_{\Delta}[-2i]$ to a non-zero scalar multiple of $\eta$, since the domain (and hence also codomain) of $\Hom(\cO_{\Delta},\gamma)$ is one-dimensional, so we can assume $\gamma \circ i = \eta$ (if necessary, we multiply $\gamma$ by the appropriate scalar). Hence 
\begin{equation}
\Hom(\cO_{\Delta}[-2],\eta)=\Hom(\cO_{\Delta}[-2],\gamma \circ i)=\Hom(\cO_{\Delta}[-2],\gamma) \circ \Hom(\cO_{\Delta}[-2],i)
\end{equation}
and we are done. 

Combining \eqref{eq:descrf} with \eqref{comsq}, there exists $u \in \HH^2(X)$ such that $\ob_{\cP}(u,v)=0$ (so that $\cP_R \ast v \ast \cP=\cP_R \ast \cP \ast u$) if and only if $f_{ii}=u$ for $1\leq i \leq n+1$ and $f_{ij}=0$ for $i \neq j$, if and only if $v_2=0$, and so we are done.
\end{proof}

\begin{thm}
\label{thm:pbecomess}
Let $F=\Phi_{\cP}:\cD(X) \to \cD(Y)$ be a $\bbP^n$-functor, and $j:Y \hookrightarrow \cY$ a one-parameter deformation of $Y$ with Kodaira--Spencer class $\kappa(\cY) \in \H^1(Y,\cT_Y)$. Assume furthermore that $\HH^{2n+2}(X)=0$. Then for every $u \in \HH^2(X)$, the obstruction class
\begin{equation}
\label{eq:obzero}
\ob_{\cP}(u,I_Y(\kappa(\cY)) \neq 0
\end{equation}
if and only if the functor 
\begin{equation}
j_*F:\cD(X) \to \cD(\cY)
\end{equation}
is an $\bbS^{2n+1}$-functor.
\end{thm}
\begin{proof}
We first check condition \eqref{cond-1} in Definition \ref{def:sph}. There are adjunctions $Lj^* \dashv j_*F \dashv Rj^!$, and since $F$ is a $\bbP^n$-functor, we have
\begin{align}
Rj^! &\simeq Lj^![-2n] \\
& \simeq Lj^* \otimes \omega_j[-2n-1] \\
& \simeq Lj^* \otimes \cO_Y(Y)[-2n-1] \\
& \simeq Lj^*[-2n-1].
\end{align}

Now we check condition \eqref{cond-2} in Definition \ref{def:sph}. By combining \cite[Corollary 11.4(ii)]{huybrechts2006fourier} with \cite[Proposition 3.1]{huybrechts2006pobjects}, and using that the normal bundle $\cN_{X \times Y/X \times \cY} \simeq \cO_{X \times Y}$ is trivial, there exists a distinguished triangle: 
\begin{equation}
\label{eq:obs-triangle}
\cP[1] \to i^*i_*\cP\to\cP \xrightarrow{\ob(\cP)}\cP[2].
\end{equation}
Next we observe that \eqref{eq:todasquare} gives:
\begin{align}
\ob(\cP)&=(1_\cP\otimes p_Y^*\kappa(\cY))\circ \at_Y(\cP)\\
&=I_Y(\kappa(\cY)) \ast \cP.
\end{align}
Applying the (contravariant) convolution functor $\cP_R \ast -$ to \eqref{eq:obs-triangle} hence yields the triangle:
\begin{equation}
\label{eq:sheaf-triangle}
\cP_R \ast \cP \to \cP_R \ast (i^*i_*\cP) \to \cP_R \ast \cO_{\Delta} \ast \cP[-1] \xrightarrow{\cP_R \ast I_Y(\kappa(\cY)) \ast \cP} \cP_R \ast \cO_{\Delta} \ast \cP[1]
\end{equation}
in $\cD(X \times X)$. Using Lemma \ref{prop:notinimage}, assumption \eqref{eq:obzero} holds if and only if the morphisms 
\begin{equation}
\cH^{2k}(\cP_R \ast I_Y(\kappa(\cY)) \ast \cP):\cH^{2k}(\cP_R \ast \cO_{\Delta} \ast \cP) \to \cH^{2k+2}(\cP_R \ast  \cO_{\Delta} \ast \cP)
\end{equation}
are isomorphisms for $0\leq k \leq n-1$. Using the isomorphism \eqref{eq:gammiso}, we hence conclude that
\begin{equation}
\cH^i(\cP_R \ast (i^*i_*\cP))\simeq
\begin{cases}
\cO_{\Delta} & \text{for } i=0, 2n+1, \\
0 & \text{otherwise.} 
\end{cases}
\end{equation}

By assumption, we have $\HH^{2n+2}(X)=0$, and so we can deduce from Lemma \ref{lem:formal} that 
\begin{equation}
\cP_R \ast (i^*i_*\cP) \simeq \cO_{\Delta} \oplus \cO_{\Delta}[-2n-1].
\end{equation}
Finally, we compute:
\begin{align}
\cP_R \ast (i^*i_*\cP) & = p_{13*}(p_{12}^*i^*i_*\cP \otimes p_{23}^*\cP_R) \\
& \simeq p_{13*}((i \times \id)^*q_{12}^{*}i_*\cP \otimes p_{23}^*\cP_R) \\
& \simeq p_{13*}(i \times \id)_*((i \times \id)^*q_{12}^{*}i_*\cP \otimes p_{23}^*\cP_R) \\
& \simeq p_{13*}(q_{12}^{*}i_*\cP \otimes (i \times \id)_*p_{23}^*\cP_R) \\
& \simeq p_{13*}(q_{12}^{*}i_*\cP \otimes q_{23}^{*}(j \times \id)_*\cP_R) \\
& \simeq ((j \times \id)_*\cP_R) \ast (i_*\cP),
\end{align}
where $p_{ij}$ (respectively $q_{ij}$) are the projections from $X \times Y \times X$ (respectively $X \times \cY \times X$) onto the factors. One then checks that $i_*\cP$ is the Fourier--Mukai kernel for $j_*F$, and 
\begin{equation}
(i_*\cP)_R\simeq(j \times \id)_*\cP_R,
\end{equation}
so $j_*F$ indeed satisfies Definition \ref{def:sph} and is hence an $\bbS^{2n+1}$-functor. If assumption \eqref{eq:obzero} does not hold, then one sees in the same way that $\cP_R \ast (i^*i_*\cP)$ does not have the correct cohomology sheaves and hence $j_*F$ is not spherical.
\end{proof}

We have used the following lemma, which can be proved by induction on the number of non-zero cohomology sheaves of $\cE$.

\begin{lem}\cite[Lemma 2.11]{markman2011integral}
\label{lem:formal}
Let $S$ be a scheme and $\cE$ an object in $\cD(S)$. Assume that 
\begin{equation}
\Ext_S^{i+1}(\cH^j(\cE),\cH^{j-i}(\cE))=0,
\end{equation}
for all $j \in \bbZ$ and $i>0$. Then $\cE$ is formal, i.e. it is isomorphic to $\bigoplus_j\cH^j(\cE)[-j]$ in $\cD(S)$. 
\end{lem}

\begin{exa}
If $X=\Spec(\bbC)$ then Theorem \ref{thm:pbecomess} specialises to \cite[Proposition 1.4]{huybrechts2006pobjects}. Indeed, in  that case $F=\cP \otimes -: \cD(\Spec(\bbC)) \to \cD(Y)$, for $\cP$ a $\bbP^n$-object in $\cD(Y)$. Since $\HH^2(\Spec(\bbC))=\HH^{2n+2}(\Spec(\bbC))=0$, condition \eqref{eq:obzero} reduces to checking that 
\begin{equation}
\ob_{\cP}(0,I_Y(\kappa(\cY)))=I_Y(\kappa(\cY)) \ast \cP  =(1_\cP\otimes\kappa(\cY)) \circ \at(\cP) \neq 0,
\end{equation}
where the second equality follows from diagram \eqref{eq:todasquare}. This is exactly the condition in \cite[Proposition 1.4]{huybrechts2006pobjects} for the pushforward $j_*\cP \in \cD(\cY)$ to become a spherical object.
\end{exa}

Hence, all examples discussed in \cite{huybrechts2006pobjects} also apply here. To give examples of Theorem \ref{thm:pbecomess} for $X\neq \Spec(\bbC)$, we need to be able to check condition \eqref{eq:obzero}, which is not easy in practice. In Section \ref{sec:hilb} we will work out a non-trivial example based on the work of Markman and Markman--Mehrotra.

\begin{rmk}
Since pushforward along the inclusion $j:Y\hra\cY$ is a spherical functor, another way to view Theorem \ref{thm:pbecomess} is that it provides a criterion for when the composition of a (specific) spherical functor with a (split) $\bbP^n$-functor is again spherical. It would be interesting to know if such a criterion exists for more general spherical and $\bbP^n$-functors.
\end{rmk}

\section{Intertwining symmetries}

In this section, we explain how \cite[Proposition 2.7]{huybrechts2006pobjects} generalises to our setting. We first give a reminder of the spherical twist (respectively $\bbP$-twist) associated to a $\bbS$-functor (respectively $\bbP$-functor). For a Fourier--Mukai functor $F=\Phi_{\cP}:\cD(X) \to \cD(Y)$, we will always denote by 
\begin{equation}
\epsilon:\cP \ast \cP_R \to \cO_{\Delta_Y}
\end{equation}
the counit of the adjunction.

\begin{prop}\cite{addington2011new,MR3692883,kuznetsov2015fractional,meachan2016note,rouquier2006categorification}
If $F=\Phi_{\cP}:\cD(X) \to \cD(Y)$ is a split $\bbS$-functor, then the functor 
\begin{equation}
T_{F}:=\Phi_{\cS}:\cD(Y) \to \cD(Y)
\end{equation}
associated to the kernel $\cS:=\cone(\epsilon) \in \cD(Y \times Y)$ is an autoequivalence.
\end{prop}

The main technical tool which allows us to follow the proof of \cite[Proposition 2.7]{huybrechts2006pobjects} is the following uniqueness result by Anno and Logvinenko.

\begin{prop}\cite[Theorem 3.1]{anno2017on}
\label{prop:Ptwist}
Let $F=\Phi_\cP:\cD(X)\to\cD(Y)$ be a Fourier--Mukai functor and suppose we have a map $f:\cP\ast\cP_R[-2]\to\cP\ast\cP_R$ such that $\eps\circ f=0$. Then all convolutions of the complex
\begin{equation}
\label{eq:Pkernel}
\cP \ast \cP_R[-2] \xrightarrow{f} \cP \ast \cP_R \xrightarrow{\epsilon} \cO_{\Delta_Y}
\end{equation}
are isomorphic. 
\end{prop}

In order to define the $\bbP$-twist associated to a $\bbP$-functor $F$, one needs to fix a choice of morphism $f \in \Hom(\cP\ast\cP_R[-2],\cP \ast \cP_R)$. Different authors make slightly different choices (compare \cite[\S4.3]{addington2011new} and \cite[\S6.2]{cautis2012flops}). To obtain the most direct generalisation of \cite[Proposition 2.7]{huybrechts2006pobjects}, we follow Cautis \cite[\S 6.2]{cautis2012flops}.

\begin{thm}
\label{th:main2}
Let $F=\Phi_{\cP}:\cD(X) \to \cD(Y)$ be a $\bbP^n$-functor, and $j:Y \hookrightarrow \cY$ a one-parameter deformation of $Y$ with Kodaira--Spencer class $\kappa(\cY) \in \H^1(Y,\cT_Y)$. Assume furthermore that $\HH^{2n+2}(X)=0$, and for every $u \in \HH^2(X)$, the obstruction class
\begin{equation}
\ob_{\cP}(u,I_Y(\kappa(\cY)) \neq 0.
\end{equation}
Setting 
\begin{equation}
f:=I_Y(\kappa(\cY)) \ast \cP \ast \cP_R - \cP \ast \cP_R \ast I_Y(\kappa(\cY)),
\end{equation}
the convolution $\cQ \in \cD(Y \times Y)$ of \eqref{eq:Pkernel} is unique and gives rise to an autoequivalence $P_{F}=\Phi_\cQ:\cD(Y)\to\cD(Y)$, defined up to natural isomorphism, which we call the $\bbP$-twist. Moreover, there is a 2-commutative diagram: 
\begin{equation}
\label{eq:2com}
\begin{tikzcd}
\cD(Y)\ar[r,"j_*"]\ar[d,"P_F"'] & \cD(\cY)\ar[d,"T_{j_*F}"]\\ 
\cD(Y)\ar[r,"j_*"] & \cD(\cY).
\end{tikzcd}
\end{equation}
\end{thm}
\begin{proof}
Uniqueness of the convolution follows from Proposition \ref{prop:Ptwist} since $\epsilon \circ f=0$. By Lemma \ref{prop:notinimage}, the Fourier--Mukai kernel $\cP$ of the $\bbP^n$-functor $F$ together with $I_Y(\kappa(\cY))$ satisfy Cautis' definition of a $\bbP$-functor \cite[p.26]{cautis2012flops}, so Proposition 6.6 in loc.cit.\ ensures that $\Phi_{\cQ}$ defines an autoequivalence (note that his assumption $\HH^1(X)=0$ was only used to obtain uniqueness of the convolution, which we obtained via Proposition \ref{prop:Ptwist}).

To show that \eqref{eq:2com} 2-commutes, we can now follow the proof of \cite[Proposition 2.7]{huybrechts2006pobjects}, keeping in mind that one has to:
\begin{enumerate}
\item replace their use of $\cE^{\vee} \boxtimes \cE$ with $\cP \ast \cP_R$,
\item replace their use of \cite[Lemma 2.1]{huybrechts2006pobjects} by Proposition \ref{prop:Ptwist},
\item replace their use of the equality $\id \boxtimes \overline{h} - \overline{h}^{\vee} \boxtimes \id=(1_{\cE^{\vee} \boxtimes \cE}\otimes\kappa(\cY \times_C \cY)) \circ \at(\cE^{\vee} \boxtimes \cE)$ by the equality $f=(1_{\cP \ast \cP_R}\otimes\kappa(\cY \times_C \cY)) \circ \at(\cP \ast \cP_R)$. 
\end{enumerate}
The last equality follows from the following computation:
\begin{align}
&(1_{\cP \ast \cP_R}\otimes\kappa(\cY \times_C \cY)) \circ \at(\cP \ast \cP_R) \\ 
&= (1_{\cP \ast \cP_R}\otimes(p_2^*\kappa(\cY)+p_1^*\kappa(\cY))) \circ \exp(\at(\cP \ast \cP_R)) \\
&=  I_Y(\kappa(\cY))\ast \cP\ast \cP_R + \cP \ast \cP_R \ast \sigma_*I_Y(\kappa(\cY)) \\
&= I_Y(\kappa(\cY)) \ast \cP \ast \cP_R-\cP \ast \cP_R  \ast I_Y(\kappa(\cY))   \\
&= f
\end{align}
where we used Proposition \ref{prop:todalem5.8} in the third line. 
\end{proof}

\begin{rmk}
A possible alternative approach to Theorem \ref{th:main2} proceeds as follows: if we set $v=I_Y(\kappa(\cY))$ then \cite[Theorem 1.1]{toda2009deformations} shows that the $\bbP$-twist $P_F$ deforms to give a 2-commuting diagram
\begin{equation}
\begin{tikzcd}
\cD(Y) \ar{d}{P_F} \ar{r}{j_*} & \cD(Y,v) \ar{d}{\widetilde{P}_F} \\
\cD(Y) \ar{r}{j_*} & \cD(Y,v)
\end{tikzcd}
\end{equation}
for some equivalence $\widetilde{P}_F$ if and only if $\ob_{\cQ}(v,v)=0$. One hence needs to check this vanishing condition, and then argue that $\widetilde{P}_F$ lifts further to an equivalence $\cD(\cY) \to \cD(\cY)$, which should coincide with $T_{j_*F}$. We have not pursued this approach further.
\end{rmk}

\section{Hilbert schemes of points on a K3 surface}
\label{sec:hilb}
The aim of this section is to show that Theorem \ref{thm:pbecomess} applies to the functor 
\begin{equation}
F=\Phi_{\cP}:\cD(X) \to \cD(M),
\end{equation}
where $X$ is a smooth projective K3 surface, $M=\Xn$ denotes the (smooth projective) Hilbert scheme of $n+1$ points on $X$ where $n>1$, and $\cP \in \Coh(X \times M)$ is the ideal sheaf of the universal closed subscheme of $X \times M$. The first input is the following result by Addington and Markman--Mehrotra.

\begin{thm}\cite{addington2011new,markman2011integral}
\label{th:addington}
The functor $F:\cD(X) \to \cD(M)$ is a $\bbP^{n}$-functor. In particular, there is an isomorphism of kernels
\begin{equation}
\label{eq:RF}
\gamma:\bigoplus_{i=0}^{n} \cO_{\Delta}[-2i]\to \cP_R \ast \cP
\end{equation}
in $\cD(X \times X)$.
\end{thm}

The next step is to find a deformation $j:M \hookrightarrow \cM$ of the Hilbert scheme, such that for all $u \in \HH^2(X)$, the obstruction 
\begin{equation}
\ob_{\cP}(u,I_{M}(\kappa(\cM))) \neq 0,
\end{equation}
where $\kappa(\cM)$ is the Kodaira--Spencer class associated to $\cM$. 

\begin{prop}
\label{prop:mm-iso}
The morphism
\begin{equation}
-\ast\cP: \HH^2(M) \to \Ext^2_{X \times M}(\cP,\cP)
\end{equation}
is an isomorphism, and the morphism
\begin{equation}
\cP \ast -: \HH^2(X) \to \Ext^2_{X \times M}(\cP,\cP)
\end{equation}
is injective.
\end{prop}
\begin{proof}
The first part is essentially contained in \cite[\S 7.2]{markman2011integral}. For the convenience of the reader we briefly sketch the argument. Consider the triangle determined by the counit 
\begin{equation}
\label{eq:counit}
\cE[1] \to \cP\ast\cP_R \xrightarrow{\epsilon} \cO_{\Delta_M}
\end{equation}
and apply $\Hom(-,\cO_{\Delta_M})$ to this triangle to obtain a long exact sequence:
\begin{equation}
\label{eq:triangle-2}
\cdots \to 
 \Hom(\cE,\cO_{\Delta_M}) \to \HH^2(M) \to \Hom(\cP\ast\cP_R,\cO_{\Delta_M}[2]) \to \Hom(\cE,\cO_{\Delta_M}[1]) 
\to \cdots
\end{equation}
It turns out that $\cE \in \cD(M \times M)$ is a sheaf, whose properties are described in \cite[Proposition 4.1]{markman2011beauville}. These properties allow one to show that in the spectral sequence 
\begin{equation}
E_2^{p,q}=\Ext^p_{M}(\cH^{-q}(\Delta_M^*\cE),\cO_M) \Rightarrow \Ext^{p+q}_{M \times M}(\cE,\cO_{\Delta_M})
\end{equation}
the terms $E_2^{0,0}=E_2^{0,1}=E_2^{1,0}=0$, and hence  
\begin{equation}
\Hom_{M \times M}(\cE,\cO_{\Delta_M}) =\Hom_{M \times M}(\cE,\cO_{\Delta_M}[1])=0.
\end{equation}
So we see that 
\begin{equation}
\HH^2(M) \to \Hom(\cP\ast\cP_R,\cO_{\Delta_M}[2])\simeq \Hom(\cP_R,\cP_R[2])\simeq \Ext^2_{X \times M}(\cP,\cP)
\end{equation}
in \eqref{eq:triangle-2} is an isomorphism.

The second part of the statement follows from the fact that $F$ is faithful, because of the factor $\cO_{\Delta}$ appearing in the isomorphism $\gamma$.
\end{proof}

In particular, by combining Proposition \ref{prop:mm-iso} with \eqref{eq:corr}, there is an induced injective morphism on the second Hochschild cohomology groups:
\begin{equation}
\label{eq:map-hh}
\phi:(-\ast \cP)^{-1} \circ (\cP \ast -):\HH^2(X) \hookrightarrow \HH^2(M),
\end{equation}
and we see that
\begin{equation}
\label{eq:image}
\bigg( \forall u \in \HH^2(X): \ob_{\cP}(u,I_{M}(\kappa(\cM))) \neq 0 \bigg) \iff I_{M}(\kappa(\cM)) \notin \im(\phi).
\end{equation}
So to determine whether Theorem \ref{thm:pbecomess} applies, we need to find (geometric) deformations of $M$ which do not lie in the image of $\phi$.

\subsection{The image of $\phi$}

A detailed study of the image of $\phi$ was undertaken in \cite[\S 7]{markman2011integral}, from which we extract the results necessary in order to apply Theorem \ref{thm:pbecomess}.

Recall that the Yoneda composition defines the structure of a graded ring on the Hochschild cohomology $\HH^{\bullet}(M)$, and that Hochschild  homology 
\begin{equation}
\HH_{\bullet}(M):=\Hom_{M \times M}(\Delta_{!}\cO_{M}[\bullet],\cO_{\Delta})
\end{equation}
is a graded module over $\HH^{\bullet}(M)$, again via Yoneda composition. In particular, for every $\xi\in\HH^i(M)$ we have an associated action map:
\begin{equation}
\label{eq:mxi}
m_{M}(\xi):\HH_j(M)\to\HH_{j-i}(M)\;;\, \zeta\mapsto\xi\circ\zeta.
\end{equation}
Using \eqref{eq:mxi}, we can define the annihilator of a subset $\Sigma\subset\HH_0(M)$:
\begin{equation}
\ann(\Sigma)=\{\xi\in\HH^2(M)\mid m_{M}(\xi)(\zeta)=\xi\circ\zeta=0\text{ for all }\zeta\in \Sigma\}.
\end{equation}
Remember there is a Chern character map 
\begin{equation}
\ch:K_0(M) \to \HH_0(M),
\end{equation}
which coincides with the usual Chern character after composing with the HKR-isomorphism \cite[Theorem 4.5]{caldararu2005mukai2}.
Now let $m\in M$ be a closed point and define the following classes:
\begin{align}
\alpha&=\ch(FR(\cO_m)), \\ 
\beta&=\ch(\cO_m)
\end{align}
in $\HH_0(M)$. We will sometimes identify $\alpha$ and $\beta$ with their images in $\bigoplus_{i\geq 0} \H^i(M,\Omega^i_M)$ under the HKR-isomorphism, so we can speak about their rank and first Chern class.

\begin{prop}
\label{cor:comm-defs}
If $c_1(\alpha) \neq 0$ then there exists $
\kappa \in\H^1(M,\cT)$ such that $I_{M}(\kappa)\notin\im(\phi)$. 
\end{prop}
\begin{proof}
From \cite[Lemma 7.9]{markman2011integral}, which is based on the proof of \cite[Proposition 6.1]{addington2014hodge}, we know that for every $\lambda \in \HH^2(X)$, there is a commuting diagram
\begin{equation}
\begin{tikzcd}
\HH_i(M) \ar{r}{R_*} \ar[swap]{d}{m_{M}(\phi(\lambda))} & \HH_i(X) \ar{r}{F_*} \ar{d}{m_X(\lambda)} & \HH_i(M) \ar{d}{m_M(\phi(\lambda))}\\
HH_{i-2}(M) \ar{r}{R_*} & \HH_{i-2}(X) \ar{r}{F_*} & \HH_{i-2}(M),
\end{tikzcd}
\end{equation}
where $F_*:\HH_{\bullet}(X) \to \HH_{\bullet}(M)$ and $R_*:\HH_{\bullet}(M) \to \HH_{\bullet}(X)$ are the maps induced on Hochschild homology by functoriality. From this diagram we immediately deduce that 
\begin{equation}
\im(\phi) \subset \{\xi \in \HH^2(M) \mid m_{M}(\xi) F_* R_* = F_*  R_*  m_{M}(\xi)\}.
\end{equation}
So to prove the proposition, it suffices to exhibit a class $\kappa \in \H^1(M,\cT)$ such that 
\begin{equation}
\label{eq:noncommute}
m_{M}(I_M(\kappa)) F_* R_* \neq F_*  R_*  m_{M}(I_M(\kappa)).
\end{equation}

As is the case for any hyperk\"ahler variety, we have a perfect pairing: 
\begin{equation}
\H^1(M,\cT)\otimes \H^1(M,\Omega)\to\H^2(M,\cO)\simeq\bbC,
\end{equation}
given by the cup-product. By assumption we have $c_1(\alpha)\neq0$ and so there exists some $0\neq\kappa \in\H^1(M,\cT)$ such that $\kappa \cdot c_1(\alpha) \neq 0$. We claim that any such $\kappa$ satisfies \eqref{eq:noncommute}.

Indeed, by \eqref{eq:RF}, $R_*F_*$ is multiplication by $n+1$, and so  
\begin{equation}
(F_*R_*)^2=(n+1)F_*R_*.
\end{equation}
Thus, we see that $\im(F_*)$ and $\ker(R_*)$ are the eigenspaces of $F_*R_*$ with eigenvalues $n+1$ and $0$, respectively. If we assume that $\xi:=I_M(\kappa)$ does not satisfy \eqref{eq:noncommute}, then $\im(F_*)$ and $\ker(R_*)$ are hence invariant with respect to the action of $m_M(\xi)$. In particular, since $\rk(\beta)=0$, we find that 
\begin{equation}
\label{eq:point}
m_M(\xi)(\alpha -(n+1)\beta)=m_M(\xi)(\alpha).
\end{equation} 
Now $\alpha-(n+1)\beta \in \ker(R_*)$, since
\begin{align}
R_*(\alpha-(n+1)\beta) &=R_*(\ch(FR(\cO_m)))-(n+1)R_*(\ch(\cO_m)) \\
&=R_*F_*(\ch(R(\cO_m)))-(n+1)\ch(R(\cO_m))\\
&=(n+1)(\ch(R(\cO_m)))-(n+1)\ch(R(\cO_m))\\
&=0,
\end{align}
and so by invariance, the left-hand side of \eqref{eq:point} also sits in $\ker(R_*)$. Moreover 
\begin{align}
\alpha&=\ch(FR(\cO_m)) \\
&=F_*\ch(R(\cO_m))
\end{align}
sits in $\im(F_*)$, so by invariance again, the same is true for the right-hand side of \eqref{eq:point}. Since $\ker(R_*)$ and $\im(F_*)$ are eigenspaces for disctinct eigenvalues, it suffices to show that $m_M(\xi)(\alpha) \neq 0$, but this is clear since we chose $\kappa\in\H^1(M,\cT)$ such that $\kappa \cdot c_1(\alpha) \neq 0$, and so we have $\xi=I_M(\kappa) \notin \ann(\alpha)$.
\end{proof}

\begin{thm}
\label{thm:hilb}
Up to algebraisation, there exists a one-parameter deformation $j:M \hookrightarrow \cM$ of $M$ such that 
\begin{equation}
j_*F:\cD(X) \to \cD(\cM)
\end{equation}
is spherical.
\end{thm}
\begin{proof}
It suffices to show that there exists some $\kappa \in\H^1(M,\cT)$ such that $I_{M}(\kappa)\notin\im(\phi)$. Indeed, since $M$ is hyperk\"ahler, it has unobstructed deformations \cite{iacono2010algebraic}, so every such $\kappa$ integrates to at least a formal deformation of $M$. We denote the corresponding one-parameter deformation by $\cM$. The theorem then follows by applying Theorem \ref{thm:pbecomess}, using Theorem \ref{th:addington}, \eqref{eq:image} and $\cM$ as input.

By Proposition \ref{cor:comm-defs}, to obtain such a $\kappa$, it suffices to show that $c_1(\alpha) \neq 0$. Observe that $c_1(\alpha)=-c_1(\cE_m)$ where $\cE$ is the sheaf appearing in \eqref{eq:counit}. Indeed, evaluating the triangle in \eqref{eq:counit} on a skyscraper sheaf $\cO_m$ and taking the first Chern class gives $c_1(\alpha)=c_1(FR(\cO_m))=c_1(\cO_m)-c_1(\cE_m)=-c_1(\cE_m)$. Now, by the comment immediately after \cite[Lemma 3.5]{markman2011beauville}, we have $c_1(\cE_m)=E/2 \neq 0$, where $E$ denotes the exceptional divisor coming from the blowup description of $M$.
\end{proof}

\begin{rmk}
Observe that the spherical twist $T_{j_*F} \in \Aut(\cD(\cM))$ cannot be induced by a spherical object on $\cM$. Indeed, the twist associated to a spherical object $\cE\in\cD(\cM)$ acts by $[-2n-1]$ on $\cE$ and by the identity on $\cE^\perp$ whereas $T_{j_*F}$ acts by $[-2n]$ on $\im(j_*F)$ and fixes everything in $\ker(Rj^!)$; the latter action follows from the isomorphism $T_{j_*F} \circ j_*F\simeq j_*F \circ [-2n]$ (see \cite[\S2.3]{addington2011new} or \cite[Lemma 1.4]{meachan2016note}) and the triangle defining the spherical twist $T_{j_*F}$. Since these actions cannot coincide, we have constructed a new autoequivalence of $\cD(\cM)$. See \cite[p.252]{addington2011new} and \cite[\S5]{krug2013new} for similar arguments.
\end{rmk}

\section{Some remarks}
\label{sec:con}

\subsection{Non-split $\bbP$-functors and noncommutative deformations}
In this paper, we have chosen a geometric approach based on Fourier--Mukai kernels, and only used split $\bbP^n$-functors. Recently, Anno and Logvinenko \cite{pfunctor} introduced the notion of a non-split $\bbP^n$-functor between enhanced triangulated categories. Using their theory, it should be possible to prove an abstract version (i.e. not assuming we are in an algebro-geometric setting) of Theorem \ref{thm:pbecomess}, while at the same time removing the assumption that the deformation of $Y$ is geometric. 

Indeed, all the ingredients we used in Section \ref{sec:hhfm} have analogues for DG-categories, and the proof of Theorem \ref{thm:pbecomess} is mostly formal. The one ingredient which complicates matters is the obstruction triangle \eqref{eq:obs-triangle}. A fully satisfactory abstract treatment does not seem to be available at the moment; the relevant issues are discussed in some detail in \cite[\S5]{MR3303245}, and are currently under investigation by these authors. 

\subsection{Deformations of Hilbert schemes of points and other moduli spaces}
Hilbert schemes of points for smooth projective surfaces have  a very rich deformation theory, which has been studied by several groups of authors \cite{belmans2019hilbert,belmansderived,MR2102090,MR1354269,hitchin2012deformations,kapustin2001noncommutative}.
Ideally, one would like to use these results to better understand the deformations $\cM$ from Theorem \ref{thm:hilb}. Indeed, since $h^{0,1}(X)=0$, \cite[\S 4]{hitchin2012deformations} applies and there is an exact sequence
\begin{equation}
\begin{tikzcd}
0 \ar{r} & \H^1(X,\cT_X) \ar{r} & \H^1(M,\cT_{M}) \ar{r} & \H^0(X,\wedge^2\cT_X) \ar{r} \ar[bend right,dashed]{l}{\rho} & 0,
\end{tikzcd}
\end{equation}
where the first map is the natural map induced by the relative Hilbert scheme, and $\rho$, which has an explicit geometric description, splits the sequence. This shows that every deformation of $M$ can be obtained from a (potentially noncommutative) deformation of $X$, and it would be interesting to understand how the Kodaira--Spencer classes obtained in the proof of Theorem \ref{thm:hilb} fit in. In a slightly different direction, Markman and Mehrotra do not just study Hilbert schemes, but also more general moduli spaces of stable sheaves on $X$, and most of their results, which we used in \S\ref{sec:hilb}, have analogues in this more general setting, so one could try and apply Theorem \ref{thm:pbecomess} to these moduli spaces as well.

\subsection{Higher order obstructions}
If there does exist a class $u\in \HH^2(X)$ such that $\ob_{\cP}(u,v)=0$, with $v=I_Y(\kappa(\cY))$, then by Theorem \ref{thm:def}, $F$ deforms to a functor 
\begin{equation}
F_{u,v}:\cD(X,u) \to \cD(Y,v),
\end{equation} 
and $j_*F$ is not spherical. In the object case, this situation was considered in \cite[\S 4]{toda2007on}, and it was shown that higher order obstructions can still be used to construct spherical twists on $\cD(\cY)$. Presumably this can also be generalised to the functor setting, but since we don't know of a situation where this applies, we have refrained from doing so in this paper.

\appendix

\section{Atiyah classes, HKR and the characteristic morphism}
\label{app}

Assume $F=\Phi_{\cP}:\cD(X) \to \cD(Y)$ is an arbitrary Fourier--Mukai functor between smooth projective varieties $X$ and $Y$. The goal of this appendix is to review, and slightly reformulate \cite[Lemmas 5.6, 5.7, 5.8]{toda2009deformations}. The reason for doing this carefully is that Toda assumes throughout his paper that $F$ is an equivalence of categories, and explicitly uses at least the fully faithfulness of $F$ to define morphisms $\exp(a)_X^+$ and $\exp(a)_Y^+$ (see the paragraph above \cite[Lemma 5.7]{toda2009deformations}) which are then used further on. However, this is a red herring, as we now explain. 

\begin{lem}\cite[Proposition 5.6]{toda2009deformations}
\label{lem:com}
The following diagrams are 2-commutative:
\begin{equation}
\label{spadesuit}
\begin{aligned}
\xymatrix{
\cD(X) \ar[rr]^-{\Delta_{X*}} \ar@{=}[d] && \cD(X\times X) \ar[d]^-{\cP\ast -} & \cD(Y) \ar[rr]^-{\Delta_{Y*}} \ar@{=}[d] && \cD(Y\times Y) \ar[d]^-{-\ast\cP}\\ 
\cD(X) \ar[rr]_-{p_X^*(-)\otimes\cP} && \cD(X\times Y) & \cD(Y) \ar[rr]_-{\cP\otimes p_Y^*(-)} && \cD(X\times Y),
}
\end{aligned}
\end{equation}
where $p_{X}, p_Y$ denote the projections from $X\times Y$ onto the factors.
\end{lem}
\begin{proof}
See the proof of \cite[Proposition 5.6]{toda2009deformations}.
\end{proof}

Recall that the universal Atiyah class of $Y$:
\begin{equation}
\at_Y:\cO_\Delta\to\Omega_\Delta[1]
\end{equation}
was defined in \eqref{alpha}. Consider the composition
\begin{equation}
\begin{tikzcd}[column sep=large]
\cO_{\Delta} \ar{r}{\at_Y} & \Delta_*\Omega_Y[1] \ar{r}{\at_Y \otimes p_Y^*\Omega_Y} & \Delta_*\Omega_Y^{\otimes 2}[2] \ar{r} & \cdots \ar{r} & \Delta_*\Omega_Y^{\otimes i}[i]
\end{tikzcd}
\end{equation}
and compose this with the anti-symmetrisation map $\Omega_Y^{\otimes i} \to \Omega^i_Y$, to obtain a map
\begin{equation}
\at_{Y,i}:\cO_{\Delta} \to \Delta_*\Omega^i_Y[i].
\end{equation}

\begin{dfn}
The exponential universal Atiyah class of $Y$ is the morphism
\begin{equation}
\label{eq:exp}
\exp(\at)_Y=\bigoplus_{i\geq 0}\at_{Y,i}:\cO_{\Delta} \to \bigoplus_{i \geq 0} \Delta_*\Omega_Y^i[i],
\end{equation}
where $\at_{Y,0}:=\id_Y$.
\end{dfn}

For an object $\cE \in \cD(Y)$, we apply a similar construction to the Atiyah class $\at(\cE):\cE \to \cE \otimes \Omega_Y[1]$ to obtain the exponential universal Atiyah class of $\cE$
\begin{equation}
\exp(\at(\cE)):\cE \to \bigoplus_{i \geq 0} \cE \otimes \Omega_Y^i[i].
\end{equation}

Now consider $\cP\in\cD(X\times Y)$: then the isomorphism $\Omega_{X\times Y}\simeq p_X^*\Omega_X\oplus p_Y^*\Omega_Y$ means that in the exponential of the Atiyah class
\begin{equation}
\exp(\at(\cP)):\cP\to\bigoplus_i\cP\otimes\Omega_{X\times Y}^i[i],
\end{equation}
one can take the summands:
\begin{align}
\exp(\at(\cP))_X &:\cP\to\bigoplus_i\cP\otimes p_X^*\Omega_X^i[i],\\
\exp(\at(\cP))_Y &:\cP\to\bigoplus_i\cP\otimes p_Y^*\Omega_Y^i[i].
\end{align}

We want to understand how these partial exponential Atiyah classes are related to the universal Atiyah classes on each component. Let $\sigma: X \times X \to X \times X$ denote the standard involution.

\begin{lem}\cite[Lemma 5.7]{toda2009deformations}
\label{lem:exp}
For $\cP \in \cD(X \times Y)$, there are equalities
\begin{align}
\label{eq:exp1} \exp(\at(\cP))_X&=\cP \ast \sigma_*\exp(\at)_X  \\
\label{eq:exp2} \exp(\at(\cP))_Y&=\exp(\at)_Y \ast \cP 
\end{align}
of morphisms in $\cD(X \times Y)$.
\end{lem}
\begin{proof}
Even though the statement of \cite[Lemma 5.7]{toda2009deformations} involves the morphisms $\exp(\at)_X^+$ and $\exp(\at)_Y^+$, which are only guaranteed to exist if $F$ is fully faithful, the proof actually shows that the equalities \eqref{eq:exp1} and \eqref{eq:exp2} hold, and works for any kernel. The only external ingredient that is used in Toda's proof is Lemma \ref{lem:com}.
\end{proof}

\begin{prop}\cite[Lemma 5.8]{toda2009deformations}\label{prop:todalem5.8}
The following diagrams commute:
\begin{equation}
\begin{tikzcd}[column sep=large]
\HT^2(X) \ar{d}{p^*_X} \ar{rr}{\sigma_*I_X} && \HH^2(X) \ar{d}{\cP \ast -}\\
\HT^2(X \times Y) \ar{rr}{ - \circ \exp(\at(\cP))} && \Ext^2_{X \times Y}(\cP,\cP) 
\end{tikzcd}
\end{equation}
\begin{equation}
\label{eq:todasquare}
\begin{tikzcd}[column sep=large]
\HT^2(Y) \ar{d}{p^*_Y} \ar{rr}{I_Y} && \HH^2(Y) \ar{d}{-\ast \cP}\\
\HT^2(X \times Y) \ar{rr}{ - \circ \exp(\at(\cP))} && \Ext^2_{X \times Y}(\cP,\cP) 
\end{tikzcd}
\end{equation}
\end{prop}
\begin{proof}
Let us focus on the second diagram; commutativity of the first one follows from a similar argument. By \cite[Proposition 4.4]{caldararu2005mukai2} (see also \cite[Proposition 5.3]{toda2009deformations} and the surrounding discussion), for $\nu \in \HT^2(Y)$ we have that $I_Y(\nu)$ is the composition:
\begin{equation}
\cO_{\Delta_Y} \xrightarrow{\exp(\at)_Y} \oplus_{i \geq 0}\Delta_{Y*}\Omega_Y^i[i] \xrightarrow{\Delta_{Y*}(\nu)} \cO_{\Delta_Y}[2]
\end{equation}
and so $I_Y(\nu) \ast \cP$ is the composition:
\begin{equation}
\cP \xrightarrow{\exp(\at(\cP))_Y} \oplus_{i \geq 0}\cP \otimes p_Y^*\Omega^i_{Y}[i] \xrightarrow{1_\cP \otimes p_Y^*(\nu)} \cP[2],
\end{equation}
where we used Lemma \ref{lem:exp} for the first morphism and Lemma \ref{lem:com} for the second. But this composition is equal to 
\begin{equation}
\cP \xrightarrow{\exp(\at(\cP))} \oplus_{i \geq 0}\cP \otimes \Omega^i_{X \times Y}[i] \xrightarrow{1_\cP \otimes p_Y^*(\nu)} \cP[2],
\end{equation}
and we are done.
\end{proof}

Finally, we show that the characteristic morphism is compatible with the HKR-isomorphism in the following sense.

\begin{prop}
\label{prop:char}
For $\cE \in \cD(Y)$, there is a commuting triangle
\begin{equation}
\begin{tikzcd}[column sep=huge]
\HT^2(Y) \ar{r}{-\circ \exp(\at(\cE))} \ar[swap]{d}{I_Y} & \Hom_Y(\cE,\cE[2]) \\
\HH^2(Y) \ar[swap]{ru}{\chi_{\cE}} &
\end{tikzcd}
\end{equation}
\end{prop}
\begin{proof}
For $\omega \in \HT^2(Y)\cong \Hom(\oplus_{i \geq 0} \Omega_Y^i,\cO_Y[2-i])$, we find
\begin{align}
\chi_{\cE}(I_Y(\omega)) &= \chi_{\cE}(\Delta_*(\omega) \circ \exp(\at)_Y) \\
&= (\Delta_*(\omega) \circ \exp(\at)_Y) \ast \cE \\
&= (\Delta_*(\omega) \ast \cE) \circ (\exp(\at)_Y \ast \cE) \\
&= \omega \circ \exp(\at(\cE))
\end{align} 
where we used C{\u{a}}ld{\u{a}}raru's \cite{caldararu2005mukai2} description of the HKR-isomorphism in the first line and Lemmas \ref{lem:com} and \ref{lem:exp} (for $X=\Spec(\bbC)$) in the last.
\end{proof}

\bibliographystyle{alpha}
\bibliography{ref}
\end{document}